\documentclass[12pt]{amsart}
\usepackage[utf8]{inputenc} \usepackage[T1]{fontenc}
\usepackage{lmodern}
\usepackage[margin=2.5cm, centering]{geometry}

\usepackage[english]{babel} 
\usepackage{geometry}
\usepackage{caption}
\usepackage{subcaption}
\usepackage{mathrsfs,amsmath,amsfonts,amstext,amssymb,amsthm,amsopn,amsthm,color, dsfont,bm,bbm}
\usepackage{enumitem}
\usepackage[titletoc]{appendix}
\usepackage[dvipsnames]{xcolor}

\usepackage{thmtools}
\usepackage[colorlinks=true, linkcolor=blue, citecolor=black]{hyperref}
\addto\extrasenglish{}
\addto\extrasenglish{}

\declaretheorem[name=Theorem]{theorem}
\newtheorem{lemma}[theorem]{Lemma}
\newtheorem{proposition}[theorem]{Proposition}
\newtheorem{corollary}[theorem]{Corollary}

\declaretheoremstyle[bodyfont=\normalfont]{remark-style}

\declaretheorem[name={Example}, style=remark-style, sibling=theorem]{example}

\newcommand{\R}{\mathbb{R}}

\newcommand{\E}{\mathbb{E}}
\renewcommand{\P}{\mathbb{P}}

\renewcommand{\leq}{\leqslant} 

\renewcommand{\geq}{\geqslant} 
\renewcommand{\ge}{\geq}
\renewcommand{\Re}{\operatorname{Re}}

\def\({\left(} 
\def\){\right)} 
\def\[{\left[}
\def\]{\right]} 
\def\<{\langle} 
\def\>{\rangle}

\def\Ca{C_4}
\def\Cb{C_3}
\def\Cf{C_5}
\def\Cc{C_1}
\def\Cd{C_2}

\title{First exit times from a bounded interval for L\'{e}vy processes}
\author{Tomasz Grzywny}
\thanks{The research was partially supported by  National Science Centre (Poland): grant 2015/17/B/ST1/01043}
\address{
	Tomasz Grzywny\\
	 {Wydzia{\l}} Matematyki\\
	Politechnika {Wroc{\l}awska}\newline
	Wyb. {Wyspia\'{n}skiego} 27,
	50-370 {Wroc\l{}aw}\\
	Poland}
\email{tomasz.grzywny@pwr.edu.pl}
\date{}
\begin{document}
\begin{abstract}
In this paper we study the mean of the first exit time from a bounded interval of  various L\'evy processes. 
We establish sharp two-sided estimates of the mean  for L\'evy processes under certain condition on their characteristic exponents.  
We also study the cumulative distribution function of the supremum and infimum processes.  
Finally, we establish integral conditions that assure that the renewal function of the ladder height process is comparable with the linear one.  
\end{abstract}

\maketitle

\section{Introduction and preliminaries}

The expected exit time $x \mapsto \E^x \tau_D$ of a Markov process $X$ from an open
bounded set $D$  is a very important function for both probability and analysis  point of view and for applications in finance and insurance. This function is a barrier for $D$ (superharmonic inside and vanishes
outside, near a part of the boundary of the set) and knowing its estimates let studying the  behaviour of solutions to the Dirichlet problem \cite{MR1801253,MR3482695} and the behaviour of the underlying process near the boundary \cite{MR2006232,MR3271268}. We shall effectively estimate this function for a bounded interval and quite general L\'evy processes on the real line having paths of unbounded variation.  To this end we use the renewal functions $V$ and $\hat{V}$ of the ascending and descending  ladder-height processes. 
The idea of using $V$ in the context of potential theory comes from  \cite{MR2928332} and  turned out to be very fruitful and allowed to prove a lot of interesting results for symmetric L\'{e}vy and Markov processes.

In this work we consider a L\'{e}vy process $X_t$  in $\mathbb{R}$ that it is not a compound Poisson.
By $\psi$ we denote its L\'{e}vy-Khintchine exponent
\begin{align*}
\E e^{i X_t \xi} = e^{-t\psi( \xi )}, \hspace{0.5cm}  \xi \in \R.
\end{align*}
which is of the form \cite[Theorem 8.1]{MR3185174}
\begin{align*}
\psi({\xi}) =  \sigma^2\xi^2-i\gamma\xi + - \int_{\R}  \left( e^{i\xi x} - 1 - ix\xi\mathds{1}_{(-1,1)}(x) \right) \nu(dx), 
\end{align*}
where $\gamma \in \R$,  $\sigma \geq 0$ and a measure $\nu(dx)$, called a L\'{e}vy measure, satisfies 
\begin{align*} \hspace{1cm} \int_{\R} \left( 1 \wedge x^2 \right) \nu(dx) < \infty.
\end{align*}
Our main result (Theorem \ref{thm:main}) provide the two-sided sharp estimates of the expected the first exit time from a bounded interval
$$\tau_{(a,b)}=\inf\{t>0:X_t\notin (a,b)\},\quad a<b$$
for processes with paths of unbounded variation that oscillate at infinity. In fact we assume the weak lower scaling condition on the real part of the characteristic exponent.  
The main result may be summarized as follows
\begin{equation}\label{eq:intro}\E^x\tau_{(a,b)}\approx V(b-x)\hat{V}(x-a),\quad a<x<b,\end{equation}
where $\approx$ means that both
sides are comparable i.e. their ratio is bounded between two positive constants. The definitions of $V$ and $\hat{V}$ are provided in Section \ref{sec:fluctuation}. We should note that the renewal functions are defined
implicitly but in the considered  setting its product enjoys simple sharp estimates in
terms of more elementary functions like the Lévy–Khintchine exponent and the
Pruitt’s concentration function $h$ (Corollary \ref{cor:HapproxV}). In addition, we believe that such estimate  will be useful since  the renewal function is non-decreasing and subadditive. In fact in our settings it satisfies the weak lower scaling property (Lemma \ref{lem:Vscal}). The main tool is estimates of the cumulative distribution functions of the supremum process $\sup_{s\leq t}X_s$ and the infimum $\inf_{s\leq t}X_t$ obtained by Kwa\'{s}nicki, Ma\l{}ecki and Ryznar \cite{MR3098066}. In our setting we provide more explicit and useful form of their estimates (Theorem \ref{prop:CDF_sup_inf}).  To this end we obtain the weak scaling property and estimates for the Laplace exponents of the ascending and descending ladder time processes (Corollaries \ref{cor:kappa_sc}  and \ref{cor:kappa_est}). It was possible due to the estimates of the transition density obtained recently in \cite{2019_TG_KS}. To obtain estimates on the renewal function we provide an integral condition equivalent to existence the non-zero drift term of the ladder height process (Proposition \ref{prop:creeping}). It is not of full generality but it refutes Doney and Maller's conjecture \cite[Remark 6]{MR1894105}. The results seem to be valuable in themselves because they break a bit of the barrier of using fluctuation theory for people from potential theory and shed some light on the general situation. We use them to derive estimates and asymptotics of the tail probability of the first hitting time of a point or an interval in the forthcoming paper \cite{2019_TG_LL}. 

  The explicit formula for the mean of the first exit time from a bounded interval is very rare. It is known for strictly  $\alpha$-stable processes (\cite{MR137148,MR0102754,MR3618135}) and in this case
	\begin{equation}\label{eq:stable}\E^x\tau_{(a,b)}=c V(b-x)\hat{V}(x-a)= \frac{(b-x)^{\alpha \rho}(x-a)^{\alpha(1-\rho)}}{\Gamma(1+\alpha)},\end{equation}
	where $\rho=\P(X_1>0)$. The estimates in form of \eqref{eq:intro} for symmetric L\'{e}vy processes were obtained in \cite{MR3007664} and next using the concentration function in \cite{MR3350043}. To the author best knowledge the estimates in our generality are not known.

For $r>0$ we define  the {\it concentration function}
$$
h(r)=\frac{\sigma^2}{r^{2}}+\int_{\mathbb{R}}\left(1\wedge\frac{|x|^2}{r^2}\right)\nu(dx)\,.
$$ 
and the drift part
\begin{align}\label{def:br}
b_r=\gamma+\int_{\mathbb{R}} x \left(\mathds{1}_{(-r,r)}(x) - \mathds{1}_{(-1,1)}(x)\right)\nu(dx)\,.
\end{align}
It is easy to observe that
\begin{equation}\label{eq:h_sc_general}
\lambda^2h(\lambda r)\leq h(r),\quad r>0,\,\, 0<\lambda\leq 1.
\end{equation}
By \cite[Lemma 4]{MR3225805}
$$\frac{h(r)}{24}\leq \sup_{|x|\leq 1/r}\Re\psi(x)\leq 2 h(r),\quad r>0.$$
More properties of $h$ can be found in \cite{2019_TG_KS}. 
For $X_t$ we denote its \textit{dual process} by $\widehat{X}_t $. Notice that $ \widehat{X}_t :=2x -X_t$ with respect to $\mathbb{P}^x$ and notice that the concentration function for the dual process is the same as for $X_t$. Every functions corresponding to the dual process we indicate by $\hat{\cdot}$.  

Let $f$ be a positive function on $(0,\infty)$.  We say that
$f$ satisfies the {\it weak lower scaling condition} $f\in\mathrm{WLSC}(\alpha, \theta)$, if $\alpha>0$, $\theta>0$, such that
$$
 f(\lambda x)\geq
\theta\lambda^{\alpha} f(x),\quad \lambda\ge 1,\, x>0.
$$

Throughout the paper by $c,c_1,\dots$ we denote non-negative constants which may depend on parameters only that we indicate $c=c(\ldots)$.  The value of  constants may change from line to line in a chain of estimates.  If we use $C,C_1,\ldots$, then they are fixed.
\section{Fluctuation theory}\label{sec:fluctuation}
By $\mathbb{P}$ we denote $\mathbb{P}^0$.
Let us introduce fundamental objects of the fluctuation theory for L\'evy processes on the real line. Let $L_t$ be the local time of the process $X_t$ reflected at its supremum $M_t =  \sup_{s \leq t} X_s$ and denote by $L^{-1}_s$ the right-continuous inverse of $L_t$, the ascending ladder-time process for $X_t$. This is a (possibly killed) subordinator and $H_s = X_{L_s^{-1}} = M_{L_s^{-1}}$, called the ascending ladder-height process. The Laplace exponent of the ascending ladder process, that is, the bivariate subordinator $(L_s^{-1}, H_s), \ (s < L(\infty)) $, is denoted by $\kappa(z, t)$. By \cite[Corollary VI.10]{MR1406564},
\begin{align*}
\kappa(z, t) = c\exp \left( \int_0^\infty \int_0^\infty (e^{-s}- e^{-zs-tx})s^{-1} \P(X_s \in dx)ds \right)
\end{align*}
For the sake of simplicity assume that $c=1$. Let us denote $\widehat{\kappa}$ as the Laplace exponent  the ascending ladder process for the dual process. Since $\P(X_t=0)=0$ for every $t>0$  it is easy to see that  $\kappa(z,0)\widehat{\kappa}(z,0) = z$. Indeed 
\begin{align*}
z  = &  \exp \left(  \int_0^\infty \left (e^{-s}- e^{-zs} \right)s^{-1}ds \right) =\exp \left(  \int_0^\infty \left (e^{-s}- e^{-zs} \right)s^{-1}  \left( \P(X_s \geq 0) + \P({X}_s \leq 0) \right)ds \right)\\ =&\kappa(z,0)\widehat{\kappa}(z,0) .
\end{align*}

By $V(x) = \int_0^\infty \P (H_s \leq x)ds$ we denote the renewal function of the process $H_s$. We have  $$\mathcal{L}V(\lambda)=\frac{1}{\kappa(0,\lambda)},\quad \lambda>0.$$
We notice that $V$ is subadditive and non-decreasing, therefore 
$$V(\lambda x)\leq 2\lambda V(x),\quad x>0, \,\lambda>1.$$

Using Vigon's equation we obtain the following bound.
\begin{lemma}\label{lem:VVhat_upper}There exists a constant $\Cc$ such that for any non-Poisson L\'{e}vy processes 
$$ V(r)\widehat V(r)\leq \frac{\Cc}{h(r)},\quad r>0.$$
\end{lemma}
\begin{proof}
Let $\delta$ be a drift and  $\eta$ be a L\'evy measure of $H_t$. By Vigon's equation \cite[Theorem 16]{MR2320889}, for $x>0$,
\begin{align*}
\eta(x, \infty) = & \int_0^\infty  \nu(y+x, \infty)\widehat V(dy)
			   \geq  \int_0^{x} \widehat V(dy) \nu(2x, \infty) = \widehat V(x)\nu(2x, \infty).
\end{align*}
By subadditivity and monotonicity of $\widehat V$ we have $ \widehat{V}(t)\leq 2 \widehat{V}(x)t/x$. Therefore
\begin{align*}
\int^t_0\eta(x, \infty)dx&\geq\int^t_0\widehat V(x)\nu(2x, \infty)dx\geq \frac{\widehat V(t)}{2t}\int^t_0 x\nu(2x,\infty)dx\\&\geq \frac{\widehat V(t)}{8t}\int^t_0 x\nu(x,\infty)dx.
\end{align*}
A consequence of \cite[Proposition III.1]{MR1406564} is
$$\int^t_0\eta(x, \infty)dx\leq c \, t \kappa(0,1/t)\approx \frac{t}{V(t)},\quad t>0.
$$
Hence
\begin{align*}
c_1\geq \frac{V(t)\widehat V(t)}{t^2}\int^t_0 x\nu(x,\infty)dx
\end{align*} 
By the same argument for $\widehat{H}_t$ we obtain
\begin{align*}
2c_1\geq \frac{V(t)\widehat V(t)}{t^2}\int^t_0 x[\nu(x,\infty)+\nu(-\infty,-x)]dx=\frac{V(t)\widehat V(t)}{2}\int_{\mathbb{R}}1\wedge(|x|^2/t^2)\nu(dx).
\end{align*}
Recall that
$\sigma^2=\delta\widehat \delta$  (see \cite[Corollary 4]{MR2320889}). That is
$$\frac{\sigma^2}{t^2}= \frac{\delta\widehat \delta}{t^2}\leq \kappa(0,1/t)\widehat \kappa(0,1/t) \approx \frac{1}{V(t)\widehat V(t)}.$$
These imply the claim.
\end{proof}

Weak scaling of the Laplace exponent of a subordinator is very useful property. We examine this property for the ladder time process.
\begin{lemma}\label{lem:kappa_scal}
Assume that there are $\rho>0$ and $T\in(0,\infty]$ such that $\P(X_t\geq0)\leq \rho$ for $t<T$,  then there is an absolute constant $c>0$ such that
$$ \kappa(\lambda z,0)\leq c\lambda^{\rho} \kappa(z,0), \quad \lambda>1,\, z\geq 1/T.$$
\end{lemma}
\begin{proof}First we consider $T=\infty$.
We have, by the Frullani integral, 
\begin{align*}
\frac{\kappa(\lambda z,0)}{\kappa( z,0)}&= \exp \left( \int_0^\infty  \left (e^{-zs}- e^{-\lambda zs} \right)s^{-1}  \P(X_s \geq 0) ds \right) \\&\leq \exp\left(\rho \int_0^\infty \left (e^{-zs}- e^{-\lambda zs} \right)s^{-1}  ds \right)=\lambda ^{\rho}.
\end{align*}
If $T<\infty$ we have for $z>1/T$
\begin{align*}
\frac{\kappa(\lambda z,0)}{\kappa( z,0)}&= 
\exp \left( \int_0^\infty  \left (e^{-zs}- e^{-\lambda zs} \right)s^{-1}  \P(X_s \geq 0) ds \right) \\&\leq \exp\left(\rho \int_0^\infty  \left (e^{-zs}- e^{-\lambda zs} \right)s^{-1}  ds + \int^\infty_T (e^{-zs}- e^{-\lambda zs} )s^{-1}  ds\right)\\&\leq \lambda ^{\rho}\exp\left(\int^\infty_Te^{-s/T}s^{-1}ds\right)=\lambda ^{\rho}\exp\left(\int^\infty_1e^{-s}s^{-1}ds\right).
\end{align*}
\end{proof}

Using the lower bound of the heat kernel obtained recently in \cite{2019_TG_KS} we obtain weak scaling properties of the Laplace exponents of ladder time processes.  
\begin{corollary}\label{cor:kappa_sc}
Assume that $\Re\psi\in \mathrm{WLSC}(\alpha,\theta)$ with $\alpha>1$, then there exists $\rho\in[1/2,1)$ such that, for $\lambda, z\geq 1$, 
\begin{align*}
 c^{-1}\lambda^{1-\rho} \kappa(z,0)\leq\kappa(\lambda z,0)& \leq c\lambda^{\rho} \kappa(z,0),  \\
 c^{-1}\lambda^{1-\rho} \widehat{\kappa}(z,0)\leq \widehat{\kappa}(\lambda z,0)& \leq c\lambda^{\rho} \widehat{\kappa}(z,0).
\end{align*}
If we additionally assume that $\mathbb{E}X_1=0$, then the above inequalities hold for every $z>0$, $\lambda\geq 1$ with $c=1$ and $\rho$ depending only on $\alpha$ and $\theta$.
\end{corollary}
\begin{proof}
By \cite[Proposition 6.1 and Remark 3.2]{2019_TG_KS} we get for $t\leq 1$, 
\begin{equation}\label{eq:prob_bounds1}\mathbb{P}(X_t>0)\geq\eta \quad \text{and}\quad   \mathbb{P}(X_t<0)\geq\eta\end{equation}
for some $\eta\in(0,1/2]$. Hence  a consequence of Lemma \ref{lem:kappa_scal} applying to $\kappa$ and $\widehat{\kappa}$ with $\rho=1-\eta$ is
\begin{align*}
 \kappa(\lambda z,0)& \leq c\lambda^{\rho} \kappa(z,0),  \\
 \widehat{\kappa}(\lambda z,0)& \leq c\lambda^{\rho} \widehat{\kappa}(z,0),
\end{align*}
for $z>1$. Since $\kappa(z,0)\widehat{\kappa}(z,0)=z$
\begin{align*} \lambda\widehat{\kappa}( z,0){\kappa}( z,0)&= \lambda z= \widehat{\kappa}(\lambda z,0){\kappa}(\lambda z,0)
\leq c \lambda^\rho \widehat{\kappa}(z,0)\kappa(\lambda z,0).
\end{align*}
If  $\mathbb{E}X_1=0$ then $b_r=\int_{|z|\geq r}zN(dz)$ (see \eqref{def:br}). Hence by \cite[Lemma 2.10]{2019_TG_KS} $t|b_{h^{-1}(1/t)}|\leq c h^{-1}(1/t)$ for all $t>0$ and one can repeat the proof of Proposition 6.1 in \cite{2019_TG_KS} to obtain \eqref{eq:prob_bounds1} for every $t>0$ with $\eta=\eta(\alpha,\theta)$. Applying Lemma \ref{lem:kappa_scal} ends the proof.
\end{proof}

\section{Main results}
In this section we present the main results of the article. At first we derive estimates of the cumulative distributions function of the supremum and infinimum processes or equivalently tails of  the distribution for the first exit time from the half-lines $(0,\infty)$ and $(-\infty,0)$.  Then we prove optimal bounds for the expected value of the first exit time from a bounded interval.

\begin{proposition}\label{prop:exit_upper}
 Let $X_t$ be a non-Poisson and $R>0$. Then, for $0<x<R$,
\begin{align*}
\E^x \tau_{(0,R)} \leq  \widehat{V}(x)V(R).
\end{align*}
\end{proposition}

\begin{proof}
According to \cite[Theorem VI.20]{MR1406564}, for any measurable function $f: [0,\infty) \to [0, \infty) $, we have 
\begin{align*}
\E^x \left[\int^\infty_0 f(X_t)dt \right] = \int_0^\infty V(dy) \int_0^x \widehat{V}(dz) f(x+y-z)
\end{align*}
We apply it to a characteristic function of an interval $[0,R]$. We have
\begin{align*}
\E^x \tau_{(0,R)} = &  \ \E^x \left[ \int_0^{\tau_{(0,R)}} \textbf{1}_{[0,R]}(X_t) dt \right] \leq \E^x \left[\int^\infty_0 \textbf{1}_{[0,R]}(X_t)dt \right]  \\=& \int_0^\infty V(dy) \int_0^x \widehat V(dz) \textbf{1}_{[0,R]}(x+y-z) 
			   \leq  \int_0^R V(dy) \int_0^x \widehat V(dz) \\=& V(R) \widehat{V}(x).
\end{align*}

\end{proof}
Combining the above proposition with Pruit's bounds of the expectation of the first exit time from a ball when the process start  from its center \cite{MR632968} we get the following result.
\begin{corollary}\label{cor:HapproxV}
Let $X_t$ be a non-Poisson. Then there exists an absolute constant $\Cd>0$ such that
$$\frac{\Cd}{h(r)+|b_r|/r}\leq  \widehat{V}(r)V(r),\quad r>0.$$
If we additionally assume that $\Re\psi\in \mathrm{WLSC}(\alpha,\theta)$ with $\alpha>1$ and $\mathbb{E}X_1=0$, then
$$\frac{1}{h(r)}\approx \widehat{V}(r)V(r),\quad r>0,$$
where the comparability constant depends only on $\alpha,\,\theta$.
\end{corollary}
\begin{proof}
By Pruitt's result   \cite{MR632968} and Proposition \ref{prop:exit_upper} $$\frac{c}{h(r)+|b_r|/r}\leq  \E^{r}\tau_{(0,2r)}\leq \widehat{V}(r)V(2r),\quad r>0.$$
The subadditivity of $V$ ends the proof of the first claim. 

If $\Re\psi\in \mathrm{WLSC}(\alpha,\theta)$ with $\alpha>1$ and $\mathbb{E}X_1=0$ by the proof of Corollary \ref{cor:kappa_sc} we get that there exists a constant $c$ such that $|b_r|\leq c\, rh(r)$, $r>0$. Hence the second claim is a consequence of Lemma \ref{lem:VVhat_upper}.
\end{proof}

Recall that, by  \cite{MR632968}, there exists an absolute constant $\Cb$ such that for any L\'{e}vy process $X_t$, for $r>0$,
\begin{align} \label{v5}
\P\left(\sup_{0\leq s \leq t} |X_s| \geq r \right)  \leq \Cb t(h(r)+|b_r|/r), \quad \P\left(\sup_{0\leq s \leq t} |X_s| \leq r\right) & \leq \frac{\Cb}{t(h(r)+|b_r|/r)}.
\end{align}

\begin{theorem}
\label{prop:CDF_sup_inf}Assume that $\Re\psi\in \mathrm{WLSC}(\alpha,\theta)$ with $\alpha>1$ and $\mathbb{E}X_1=0$, then there exists a constant $\Ca=\Ca(\alpha,\theta)$ such that, for $x,\,t>0$, 
\begin{align*}
\Ca^{-1} \min \{1, \frac{V(x)}{V(h^{-1}(1/t))} \} &\leq \P (\sup_{s\leq t}X_s < x) \leq \Ca \min \{1, \frac{V(x)}{V(h^{-1}(1/t))} \},\\
\Ca^{-1} \min \{1, \frac{\widehat{V}(x)}{\widehat{V}(h^{-1}(1/t))}\} &\leq \P (\inf_{s\leq t}X_s>- x) \leq \Ca\min\{1, \frac{\widehat{V}(x)}{\widehat{V}(h^{-1}(1/t))}\}.
\end{align*} 
\end{theorem}
\begin{proof}
The following bounds are a consequence of Corollary \ref{cor:kappa_sc} and \cite[Corollary 3.2]{MR3098066}, for $t,x>0$, 
\begin{align}\label{prop:CDF_sup_inf_1}
c_1 \min \{1, \kappa(1/t,0)V(x) \} &\leq \P (\sup_{s\leq t}X_s < x) \leq \min \{1, 2 \kappa(1/t,0)V(x) \},\\
c_1 \min \{1, \widehat{\kappa}(1/t,0)\widehat{V}(x) \} &\leq \P (\inf_{s\leq t}X_s >- x) \leq \min \{1, 2 \widehat{\kappa}(1/t,0)\widehat{V}(x) \},\nonumber
\end{align}
 where $c_1$ depends only on the scaling characteristics. Hence it remains to prove that 
$$\kappa(\lambda,0) V(h^{-1}(\lambda))\approx\widehat{\kappa}(\lambda,0) \widehat{V}(h^{-1}(\lambda))\approx 1,\quad \lambda>0.$$
By  \eqref{v5}, for $t_0= 1/\left(2\Cb (h(\lambda)+|b_\lambda|/\lambda)\right)$, we have 
\begin{align*}
\P\left(\sup_{0\leq s \leq t_0} |X_s| \geq \lambda \right)   \leq \frac{1}{2} 
\end{align*}
Hence
\begin{align*}
\frac{1}{2} \leq \P \left(\sup_{s\leq t_0} |X_s| < \lambda\right) \leq \P \left(\sup_{s\leq t_0} X_s < \lambda\right). 
\end{align*}
and $\P \left(\inf_{s\leq t_0} X_s > - \lambda\right)\geq 1/2$.
Combining these with \eqref{prop:CDF_sup_inf_1}   gives ${\kappa} (\frac{1}{t_0}, 0){V}(\lambda) \geq \frac{1}{4}$ and $\widehat{\kappa} (\frac{1}{t_0}, 0)\widehat{V}(\lambda) \geq \frac{1}{4}$. These together with Corollary \ref{cor:HapproxV} imply
$$ {\kappa}\left (\frac{1}{t_0}, 0\right){V}(\lambda)\leq 4 \widehat{\kappa} \left(\frac{1}{t_0}, 0\right)\widehat{V}(\lambda) {\kappa}\left (\frac{1}{t_0}, 0\right){V}(\lambda)=4\frac{{V}(\lambda)\widehat{V}(\lambda)}{t_0}\leq c.$$
Hence $${\kappa} \left(\frac{1}{t_0}, 0\right)\approx \frac{1}{V(\lambda)} \quad \text{and}\quad \widehat{\kappa} \left(\frac{1}{t_0}, 0\right)\approx\frac{1}{\widehat{V}(\lambda)}.$$
Since by the proof of Corollary  \ref{cor:kappa_sc} $h(\lambda)t_0\approx 1$ we obtain by Corollary \ref{cor:kappa_sc}
$${\kappa} (h(\lambda), 0)\approx \frac{1}{V(\lambda)} \quad \text{and}\quad \widehat{\kappa} (h(\lambda), 0)\approx\frac{1}{\widehat{V}(\lambda)}.$$
\end{proof}
By the above proof we have the following result.
\begin{corollary}\label{cor:kappa_est}
Assume that $\Re\psi\in \mathrm{WLSC}(\alpha)$ with $\alpha>1$ and $\mathbb{E}X_1=0$,
then
$$\kappa(\lambda,0)\approx \frac{1}{V(h^{-1}(\lambda))}, \quad \lambda>0,$$
where the comparability constant depends only on the scaling characteristics.
\end{corollary}

\begin{lemma}\label{lem:Vscal}
Assume that $\Re\psi\in \mathrm{WLSC}(\alpha,\theta)$ with $\alpha>1$ and $\mathbb{E}X_1=0$, then there exists a constant $\Cf=\Cf(\alpha,\theta)$ such that
$$V(\lambda x)\geq \Cf \lambda^{\alpha-1}V(x),\quad x>0,\,\lambda\geq 1.$$
\end{lemma}

\begin{proof}
By subadditivity of $V$,  Corollary \ref{cor:HapproxV} and comparibility of $h(r)$ with $\Re \psi(1/r)$ imply 
\begin{align*}
\widehat V(x) V(\lambda x)  \geq & \frac{1}{2\lambda} \widehat V(\lambda x) V(\lambda x) 	 \geq \frac{c}{\lambda}  \frac{1}{h(\lambda x)} \geq  \frac{c\lambda^\alpha}{\lambda} \frac{1}{h(x)} \\
				 \geq & c\lambda^{\alpha-1} \widehat V(x)V(x).
\end{align*}
That is
\begin{align*}
V(\lambda x) \geq c \lambda^{\alpha-1} V(x).
\end{align*}
\end{proof}

\begin{theorem}\label{thm:main}
  Assume that $\Re\psi\in \mathrm{WLSC}(\alpha,\theta)$ with $\alpha>1$ and $\mathbb{E}X_1=0$ then there exists a constant $C=C(\alpha,\theta)>0$ such that for any $R>0$ and $0<x<R$
\begin{align*}
C\, \widehat{V}(x)V(R-x) \leq \E^x \tau_{(0,R)} \leq 2 \widehat{V}(x)V(R-x).
\end{align*}
\end{theorem}

\begin{proof}
Fix $R>0$. First assume that $0<x\leq R/2$. The upper bound is a consequence of Proposition \ref{prop:exit_upper} and subadditivity of $V$. 
Let us notice  that the assumption $\Re\psi\in \mathrm{WLSC}(\alpha)$ with $\alpha>1$ implies that the process $X_t$ has paths of unbounded variation (see \cite[Lemma 2.9]{2019_TG_KS}). Hence, by \cite{MR0242261}, $0$ is regular for half-lines $(0,\infty)$ and $(-\infty,0)$. Therefore, by \cite[Theorem 2]{MR573292}, $\widehat{V}$ is harmonic and continuous function on $(0,\infty)$. That is $ \E^x \widehat{V}(X_{\tau_{(r,R)}}) = \widehat{V}(x) $, for $0<r<R$ and $x\in \mathbb{R}$. Since $\widehat{V}$ is locally bounded and continuous by quasi left continuity of $X_t$ and Fatou Lemma we have  $\widehat{V}(x)\geq \E^x \widehat{V}(X_{\tau_{(0,R)}})$. By monotonicity of $\widehat{V}$ we obtain
\begin{align*}
\widehat{V}(x) \geq \E^x \widehat{V}(X_{\tau_{(0,R)}})\geq \E^x \left[ \widehat{V}(X_{\tau_{(0,R)}}),  X_{\tau_{(0,R)}} \geq R  \right ] \geq  \widehat{V}(R) \P^x (X_{\tau_{(0,R)}} \geq R).
\end{align*}
 Hence,
\begin{align*}
\P^x (\tau_{(0,R)} < \tau_{(0,\infty)}) \leq \frac{\widehat{V}(x)}{\widehat{V}(R)}.
\end{align*}
Observe that by the Markov inequality, for any $t>0$,
\begin{align*}
\P^x (\tau_{(0,\infty)} \geq t) \leq \P^x (\tau_{(0,R)} \geq t)  + \P^x (\tau_{(0,R)} < \tau_{(0,\infty)}) \leq  \frac{\E^x \tau_{(0,R)}}{t} + \P^x (\tau_{(0,R)} < \tau_{(0,\infty)})
\end{align*}
Hence,  by Theorem \ref{prop:CDF_sup_inf}, for $t>0$,
\begin{equation}\label{eq:ex_1}
\E^x \tau_{(0,R)}\geq t\left (\Ca^{-1}\min\left\{1,\frac{\widehat{V}(x)}{\widehat{V}(h^{-1}(1/t))}\right\}- \frac{\widehat{V}(x)}{\widehat{V}(R)}\right).
\end{equation}
Now we choose appropriate time $t>0$  such that we get the lower bound for small $x$. 
Let us observe that by the scaling property for $h^{-1}$ and $\widehat{V}$ (see Lemma \ref{lem:Vscal}) there exists a constant $c_1\geq 4$,  depending only on $\alpha$ and $\theta$, such that for $t=(c_1h(R))^{-1}$ we have 
$$\widehat{V}(h^{-1}(1/t))\leq \frac{\widehat{V}(R)}{2\Ca}.$$
Hence by \eqref{eq:ex_1} and \eqref{eq:h_sc_general} , for $0<x\leq R/\sqrt{c_1}$ we obtain
\begin{align*}
\E^x \tau_{(0,R)} \geq &  t\left(\Ca^{-1}\frac{\widehat{V}(x)}{\widehat{V}(h^{-1}(1/t))}-\frac{\widehat{V}(x)}{\widehat{V}(R)}\right) \geq t\frac{\widehat{V}(x)}{\widehat{V}(R)}.
\end{align*}
By Corollary \ref{cor:HapproxV} we have $t\approx V(R)\widehat{V}(R)$ therefore  we infer
$$\E^x \tau_{(0,R)} \geq c \widehat{V}(x)V(R),$$
for $0<x\leq R/\sqrt{c_1}$.

Let us denote $x_0=R/\sqrt{c_1}$. Now we consider  $x_0<x\leq R/2$. Since $|b_r|/r\leq c h(r)$ we have by Pruitt's lower bound of $\E\tau_{(-r,r)}$ (\cite{MR632968})  
\begin{align*}
\E^x \tau_{(0,R)} \geq &  \E^x \tau_{(0,2x)}\geq \E^{x_0} \tau_{(0,2x_0)} \approx \frac{1}{h(x_0)} \approx \frac{1}{h(R)}\approx \widehat V(R)V(R)\geq \widehat V(x)V(R). 
\end{align*}
For $x>R/2$ one can use the above reasoning for the dual process.
\end{proof}
\section{Estimates of the renewal function}
In this section we present sharp and explicit estimates for the renewal functions in several cases. 

\begin{example}
Let $X_t$ be strictly $\alpha$-stable process ($1<\alpha<2$) then it is well known that the ladder height processes are stable subordinators and for $\rho=\P(X_1>0)$,
$$V(r)=c_1r^\alpha\rho\quad \text{and} \quad\widehat V(r)=c_2r^{\alpha(1-\rho)},\quad r>0.$$
Hence, by Theorem \ref{thm:main} we have $$\E^x\tau_{0,R}\approx x^{\alpha(1-\rho)} (R-x)^{\alpha\rho}\quad 0<x<R.$$ 
In fact  in this case the equality \eqref{eq:stable} holds instead of the comparability (see i.e. \cite[Remark 5]{MR3618135}).
\end{example}

\begin{example}
Let $X_t$ be symmetric. Then  $\hat{V}(r)=V(r)$ and by Corollary \ref{cor:HapproxV} $$V(r)\approx \frac{1}{\sqrt{h(r)}},\quad r>0.$$
In fact the above comparability holds for every non-Poisson L\'{e}vy process (see \cite[Proposition 2.4]{MR3350043}).
Theorem \ref{thm:main} implies $$\E^x\tau_{(0,R)}\approx \frac{1}{\sqrt{h(x)h(R-x)}},\quad 0<x<R.$$ 
\end{example}

Now we discuss  when the renewal function $V$ behave like the linear one.
\begin{lemma}\label{lem:ImRe}
Assume that $\Re\psi\in \mathrm{WLSC}(\alpha,\theta)$ with $\alpha>1$, then there exists $C=C(\psi)$ such that
$$|\mathrm{Im}\psi(\xi)|\leq C\Re\psi(\xi),\quad |\xi|\geq 1.$$
If additionally $\mathbb{E} X_1=0$, then  the above inequality holds for $\xi\in\mathbb{R}$ and $C=C(\alpha,\theta)$.
\end{lemma}
\begin{proof}Since $\Re\psi\in \mathrm{WLSC}(\alpha,\theta)$ with $\alpha>1$, we infer $\int_{|x|\geq1}|x|\nu(dx)<\infty$ (see \cite[Lemma 2.10]{2019_TG_KS}).
Let us notice that $$\mathrm{Im}\psi(\xi)=-(\gamma\xi+ \int_{|x|\geq 1}\sin(x\xi)\nu(dx))+\int^1_{-1}(x\xi-\sin(x\xi))\nu(dx).$$ 
Hence, for $|\xi|\geq 1$,
\begin{align*}|\mathrm{Im}\psi(\xi)|&\leq |\gamma||\xi|+\frac{1}{6}\int^{1/|\xi|}_{-1/|\xi|}|x\xi|^3\nu(dx)+\int_{|x|>1/|\xi|} |x\xi|\nu(dx)\\&\leq |\gamma||\xi|+\frac{1}{6}\int^{1/|\xi|}_{-1/|\xi|}|x\xi|^2\nu(dx)+2|\xi|\int_{1/|\xi|<|x|}|x|\nu(dx).
\end{align*}
Since $|\gamma||\xi|\leq \theta^{-1}|\gamma|h(1/|\xi|)/h(1)$ and $K(1/|\xi|)\leq 2\Re\psi(\xi)$ by \cite[Lemmas 2.10 and 2.3]{2019_TG_KS} we obtain the first claim. If $\mathbb{E}X_1=0$ we have $\gamma+\int_{|x|\geq 1}x\nu(dx)=0$. That is
$$\mathrm{Im}\psi(\xi)=\int_{\mathbb{R}}(x\xi-\sin(x\xi))\nu(dx).$$
Hence, for $\xi\neq 0$,
$$|\mathrm{Im}\psi(\xi)|\leq \frac{1}{6}\int^{1/|\xi|}_{-1/|\xi|}|x\xi|^3\nu(dx)+2|\xi|\int_{1/|\xi|<|x|}|x|\nu(dx)\leq C\,h(1/|\xi|),$$
where in the last inequality we use again  \cite[Lemma 2.10]{2019_TG_KS}.
\end{proof}

\begin{lemma}\label{lem:ex3} Assume that $\Re\psi\in \mathrm{WLSC}(\alpha,\theta)$ with $\alpha>1$ and $\E X_1=0$. Then
$$\lim_{\lambda\to0^+}\int_\R(1-\cos(xy))\Re\frac{1}{\psi(y)+\lambda}dy\approx \frac{1}{|x|h(|x|)},\quad x\neq 0.$$
\end{lemma}
\begin{proof}
By Lemma \ref{lem:ImRe}
$$\Re\frac{1}{\psi(y)+\lambda}=\frac{\Re\psi(y)+\lambda}{(\Re\psi(y)+\lambda)^2+(\mathrm{Im}\psi(y))^2}\approx \frac{1}{\Re\psi(y)+\lambda}.$$
This together with the proof of Lemma 2.14 in \cite{MR3636597} allow us to use the dominated convergence theorem to justify existence of the limit in the claim.  In addition, by the monotone convergence theorem
$$\lim_{\lambda\to0^+}\int_\R(1-\cos(xy))\Re\frac{1}{\psi(y)+\lambda}dy\approx \int_\R(1-\cos(xy))\frac{1}{\Re\psi(y)}dy.$$
The claim is a consequence of  \cite[Lemma 2.14]{MR3636597}.
\end{proof}
In our setting we are provide the integral condition equivalent to the existence of the non-zero drift term of the ladder height subordinator. This type condition was conjectured by Doney and Maller in \cite[Remark 6]{MR1894105}. In fact the denominator in our condition is greater than in Doney and Maller's proposition. We notice that one can consider a spectrally positive L\'{e}vy process with non-zero Gaussian term then our condition is finite but Doney and Maller's can not be even computed but it is well known that then the drift term is non-zero. Hence their conjecture is incorrect.  
\begin{proposition}\label{prop:creeping}Assume that $\Re\psi\in \mathrm{WLSC}(\alpha,\theta)$ with $\alpha>1$ and $\E X_1=0$.
$V(x)\approx x$ for $0<x<1$ if and only if
\begin{equation}\label{eq:123}\int^1_0\frac{\nu(y,\infty)}{h(y)}\frac{dy}{y}<\infty.\end{equation}
\end{proposition}
\begin{proof}
 Let us notice that $\int^1_0 1/(xh(x))dx<\infty$. Hence, combining Lemma \ref{lem:ex3} with Theorem 3.2 and Corollary 3.1 in \cite{MR321198} we obtain that  \eqref{eq:123} is equivalent to having a positive drift by the ladder height process. Since $V(x)\approx \frac{1}{\kappa(0,1/x)}$, having a positive drift by the ladder height process  is equivalent to $V(x)\approx x$, $0<x<1$. 
\end{proof}
Notice that $V(x)\approx x$, $x>1$ if and only if the ladder height process has finite moment $\E H_1<\infty$. 
Let us observe that the weak lower scaling property with the exponent grater than 1 and $\E X_1=0$ imply that $X_t$ oscillates at infinity (see \cite[Theorem 4.4]{MR1922446}). Therefore one could use \cite[Theorem 8]{MR1894105} to check whenever $\E H_1<\infty$ but unfortunately the proof of equivalence of (3.13) with the other conditions in  \cite[Theorem 8]{MR1894105} is incorrect, since the authors use there two facts that do not hold in their generality. Namely that $\nu(-\infty,0)>0$ (the 11th line from below on the page 208) and the support of the distribution of the process with truncated large jump it is not contained in the negative half-line (the 6th line from below on the page 208). To see that the integral condition (3.13) is improper one can consider spectrally positive L\'{e}vy process with finite second moment. Then $\E H_1<\infty$ but the integral in  (3.13) can not be even considered because the denominator is equal to $0$.  Fortunately, processes, that we are considering, have infinite variation, therefore the second fact holds in our setting always. Hence, if   $\nu(-\infty,0)>0$ one can use \cite[Theorem 8]{MR1894105} if it is not the case, the process is spectrally positive and  $\E H_1<\infty$ if and only if $\E X_1^2<\infty$. Therefore we have complete description, when $V(x)\approx x$, $x>1$. If the tail of the L\'{e}vy measure on the negative half-line dominated the tail on the positive one the integral condition is similar to the condition in the above proposition. 
\begin{proposition}
Let $\Re\psi\in \mathrm{WLSC}(\alpha,\theta)$ with $\alpha>1$ and $\E X_1=0$. Assume that there are $C,R>0$ such that
$\nu(z,\infty)\leq C \nu(-\infty,-z)$, $z>R$. Then
$V(x)\approx x$, $x>1$ if and only if
\begin{equation}\label{eq:1234}\int_1^\infty\frac{\nu(y,\infty)}{h(y)}\frac{dy}{y}<\infty.\end{equation}
\end{proposition}  
\begin{proof}
Since $\nu(z,\infty)\leq C \nu(-\infty,-z)$, $z>R$ we have, for $x>2R$ 
\begin{align*}\int^x_0\int^\infty_y\nu(-\infty,-z)dz&\approx \int^x_0\int^\infty_y\nu(\{|u|>z\})dz\\
&=\int^\infty_0|y|(|y|\wedge x)\nu(dy).
\end{align*}
Since, by \cite[Lemma 2.10]{2019_TG_KS}  $\int_{|y|>x}|y|\nu(dy)\leq c x h(x)$, the above implies 
$$c_1 (x^2h(x)-\sigma^2)\leq \int^x_0\int^\infty_y\nu(-\infty,-z)dz\leq c_2 x^2h(x),$$
which finish the proof if $\sigma=0$ due to \cite[Theorem 8]{MR1894105}. If it is not the case we have two  possibilities. The first one $\nu(0,\infty)=0$ but then $V(x)=c x$, $x\geq 0$ and \eqref{eq:1234} holds. And if    $\nu(0,\infty)>0$ we have  $x^2h(x)-\sigma^2\geq  \int_{\R}y^2\wedge 1\nu(dz)>0$, $x>1$. That is $x^2h(x)-\sigma^2\approx x^2h(x)$, $x>1$ and the claim is a consequence of \cite[Theorem 8]{MR1894105}.
\end{proof}
\begin{example}
Let $\nu(r,\infty)\leq c\nu(-\infty,-r)/\ln(r+1/r)^{1+\beta}$, $r>0$, for some $c,\,\beta>0$. If a function $x\mapsto x^2\int^{1/x}_0u\nu(-\infty,-u)du$ satisfies the weak scaling property with the exponent greater than 1, then $V(r)\approx r$, $r>0$. Hence, for any $R>0$,
$$\E^x\tau_{(0,R)}\approx \frac{x}{(R-x)h(R-x)},\quad 0<x<R.$$
\end{example}

\bibliographystyle{abbrv}

\end{document}